\documentclass[final]{siamltex}

\usepackage{amsmath, amsfonts,
  verbatim, fancyhdr, graphicx, array, subfigure,
  color, listings, afterpage, epsfig, boxedminipage, rotating, url}

\newcommand{\citet}{\cite}
\newcommand{\citep}{\cite}

\newcommand{\E}{\ensuremath{\text{E}\,}}

\newcommand{\pmat}[1]{\ensuremath{\begin{pmatrix} #1 \end{pmatrix}}}

\title{Matrix-Free Approximate Equilibration}

\author{Andrew M.~Bradley\footnotemark[1]\ \footnotemark[3]\ \footnotemark[4]
  \and Walter Murray\footnotemark[2]\ \footnotemark[3]}

\begin{document}
\maketitle

\renewcommand{\thefootnote}{\fnsymbol{footnote}}

\footnotetext[1]{Dept.~of Geophysics, Stanford University ({\tt
    ambrad@cs.stanford.edu}).}
\footnotetext[2]{Inst.~for Comp.~and Math.~Eng., Stanford University ({\tt
    walter@stanford.edu}).}
\footnotetext[3]{Supported by a National Science Foundation Graduate Research
  Fellowship and a Scott A.~and Geraldine D.~Macomber Stanford Graduate
  Fellowship.}
\footnotetext[4]{Additional support from the Office of Naval Research and the
  Army High Performance Computing Research Center.}

\begin{abstract}
The condition number of a diagonally scaled matrix, for appropriately chosen
scaling matrices, is often less than that of the original. Equilibration scales
a matrix so that the scaled matrix's row and column norms are equal. Scaling can
be approximate. We develop approximate equilibration algorithms for nonsymmetric
and symmetric matrices having signed elements that access a matrix only by
matrix-vector products.
\end{abstract}

\begin{keywords}
  binormalization, doubly stochastic, matrix equilibration, matrix-free algorithms
\end{keywords}

\begin{AMS} 15A12, 15B51, 65F35 \end{AMS}

\pagestyle{myheadings}
\thispagestyle{plain}
\markboth{A.M.~BRADLEY AND W.~MURRAY}
{MATRIX-FREE APPROXIMATE EQUILIBRATION}

\section{Introduction}
For a square, nonnegative, real, nonsymmetric matrix $B$, equilibration in the
1-norm finds $x, y > 0$ such that $X B y = e$ and $Y B^T x = e$, where $X =
\text{diag}(x)$ and similarly for other vectors, and $e$ is the vector of all
ones. Hence $X B Y$ is doubly stochastic. For a symmetric matrix, symmetric
equilibration finds $x > 0$ such that $X B x = e$. If $B = A \circ A$ for $A$ a
real, possibly signed, matrix, where $\circ$ denotes the element-wise product,
then these equations equilibrate $A$ in the 2-norm. Equilibration in the 2-norm
is often called \emph{binormalization}. Approximate equilibration scales a
matrix so that its row and column norms are almost equal. Both the exactly and
approximately equilibrated matrices often have smaller condition numbers than
the original. In this paper we always use the 2-norm condition
number. Equilibration is particularly usefully applied to matrices for which
simpler diagonal scaling methods fail: for example, to indefinite symmetric
matrices. In Section \ref{sec:dscale}, we compare equilibration with Jacobi
scaling when applied to symmetric matrices.

In some problems, accessing elements of a matrix is expensive. What are often
called \emph{matrix-free} algorithms access a matrix only by matrix-vector
products. If $A$ is a matrix having nonnegative elements, then many algorithms
already exist to equilibrate $A$ using only matrix-vector products: for example,
the Sinkhorn-Knopp iteration. But if $A$ has signed elements, then one must
obtain $|A|$ to use these algorithms, which requires accessing the elements of
$A$. In Section \ref{sec:algs}, we develop matrix-free approximate equilibration
algorithms for square nonsymmetric and symmetric matrices having signed
elements, and we report the results of numerical experiments with these
algorithms in Section \ref{sec:numexp}.

\section{Diagonal scaling of symmetric matrices} \label{sec:dscale}
\begin{figure}[tb]
\centering
\includegraphics[width=5in]{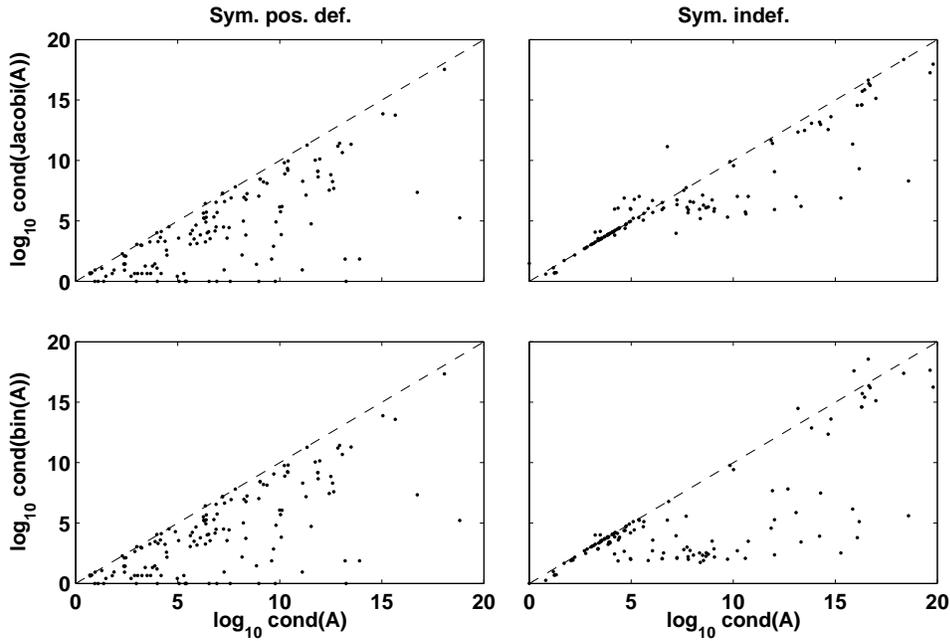}
\caption{Numerical study of conditioning of symmetric matrices. The four plots
  show condition number of the Jacobi-scaled (top) and binormalized (bottom)
  symmetric positive definite (left) and indefinite (right) matrix as a function
  of the condition number of the unscaled matrix.}
\label{fig:def}
\end{figure}

Jacobi scaling pre- and post-multiplies a square, usually symmetric positive
definite (spd) matrix by a diagonal matrix such that the scaled matrix has unit
diagonal elements.

Numerical experiments show that the condition number of the equilibrated or
Jacobi-scaled matrix is often considerably less than that of the original
matrix. Figure \ref{fig:def} shows the results of a numerical experiment using
323 symmetric matrices from the University of Florida Sparse Matrix Collection
\cite{davis_uf}; see Section \ref{sec:numexp} for further details on the data
set. The matrices used in this experiment have sizes 10 to 36441, with a median
size of 5000. Figure \ref{fig:def} shows the condition number of the scaled
matrix as a function of that of the unscaled matrix. Two diagonal scaling
methods are used: Jacobi (top) and binormalization (bottom). Matrices are
divided into positive definite (left) and indefinite (right).

This experiment shows that if a matrix is spd,
then equilibration and Jacobi scaling reduce the condition number by about the
same amount; indeed, the two corresponding plots are almost identical. It also
shows that when the two methods are applied to an indefinite matrix---in the
case of Jacobi scaling, replacing a zero diagonal element with a one---the
condition number of the Jacobi-scaled matrix is likely to be substantially
greater than that of the equilibrated matrix. For these reasons, equilibration
of symmetric indefinite matrices can be thought of as a generalization of Jacobi
scaling of spd matrices, raising the question of the relationship between the
two scaling methods when applied to spd matrices.


Let $A$ be an $n \times n$ spd matrix whose diagonal elements are all one. Let
$\kappa(\cdot)$ denote the 2-norm condition number of a matrix. Van der Sluis
showed that $\kappa(A) \le n \min_d \kappa(DAD)$ (Theorem 4.1 of
\cite{vandersluis}) and that if $A$ has at most $q$ nonzero elements in any row,
then $\kappa(A) \le q \min_d \kappa(DAD)$ (Theorem 4.3 of \cite{vandersluis}). A
matrix $C$ has \emph{Young's property A} if there exists a permutation matrix
$P$ such that
\begin{equation*}
  P C P^T = \pmat{D_1 & C_1 \\ C_2 & D_2}
\end{equation*}
and $D_1$ and $D_2$ are square diagonal matrices. Forsthye and Straus showed
that if the matrix $A$ has in addition Young's property A, then $\kappa(A) =
\min_d \kappa(DAD)$ (Theorem 4 of \cite{forsythe-straus}). In summary, these
three theorems state that Jacobi scaling is within a factor of $n$, $q$, or 1 of
optimal among all diagonal scaling matrices.

If $A$ is spd, then so is $B \equiv A \circ A$ by the Schur Product Theorem
(see, for example, Theorem 7.5.3 of \cite{haj}). Suppose $A$ has unit diagonal
elements. Then so does $B$. Moreover, $B_{ij} < 1$ for $i \ne j$.
Suppose Jacobi scaling---replacing a zero diagonal element with a one---has been
applied to an $n \times n$ symmetric matrix $\bar A$ to yield the matrix $A$,
and again let $B \equiv A \circ A$. Consider the vector of row sums $s \equiv B
e$. If $\bar A$ is indefinite, $0 \le s_i < \infty$. If $\bar A$ is spd, as
every diagonal element of $B$ is 1, $s_i \ge 1$; and as every off-diagonal
element $B_{ij} < 1$, $s_i < n$.

Let $\mu(v)$ be the mean of the elements of an $n$-vector $v$ and
$\text{var}(v)$ the variance: $\text{var}(v) \equiv n^{-1} \sum_i (v_i -
\mu(v))^2$. If a matrix is binormalized, then the variance of the vector of its
row 2-norms is 0. If $\bar A$ is indefinite, $\text{var}(s)$ can be arbitrarily
large. But if $A$ is spd, then $\text{var}(s) < (n-1)^2$. For as each $1 \le s_i
< n$, $(s_i - \mu(s))^2 < (n-1)^2$, and so $n^{-1} \sum_i (s_i - \mu(s))^2 <
n^{-1} \sum_i (n-1)^2 = (n-1)^2$.

From the other direction, an immediate corollary of inequality 2 in
\cite{livne04spd} is that if an spd matrix $\bar A$ is equilibrated in the
2-norm to form $\tilde A$, then $n^{-1/2} < \tilde A_{ii} \le 1$ (the upper
bound follows immediately from equilibration to unit row and column 1-norms); if
$\bar A$ is indefinite, then of course $-1 \le \tilde A_{ii} \le 1$.

In summary, if a matrix is spd, Jacobi scaling produces a matrix that is not
arbitrarily far from being binormalized, and binormalization produces a matrix
whose diagonal elements are bounded below and above by positive numbers. The
bounds depend on the size of the matrix. If a matrix is symmetric indefinite,
then neither statement holds.

\section{Algorithms} \label{sec:algs}
Sink\-horn and Knopp analyzed the convergence properties of the iteration
\eqref{eq:sk-nonsym}:
\begin{equation} \label{eq:sk-nonsym}
  r^{k+1} = (B c^k)^{-1}, \quad c^{k+1} = (B^T r^{k+1})^{-1}.
\end{equation}
The reciprocal is applied by element. $c^0$ is a vector whose elements are all
positive. According to Knight \cite{knight}, the iteration was used as early as
the 1930s.

Parlett and Landis \cite{parlett-landis} generalized Sinkhorn and Knopp's
convergence analysis and developed several new algorithms, one of which, EQ,
substantially outperformed the Sinkhorn-Knopp iteration on a test set. Khachiyan
and Kalantari \cite{khachiyan-kalantari} used Newton's method to scale positive
semidefinite symmetric matrices. Livne and Golub \cite{livne04bin} developed
algorithms for symmetric and nonsymmetric matrices based on the
Gauss-Seidel-Newton method. Knight and Ruiz \cite{knight-ruiz} devised an
algorithm based on an inexact Newton method that uses the conjugate gradients
iteration.

Nonuniqueness of equilibration in the infinity norm motivates multiple
algorithms that consider both efficiency and quality of the scaling under
criteria other than the infinity norms of the rows and columns. A matrix can be
scaled in the infinity norm if it has no zero rows or columns. The simplest
nonsymmetric algorithm is first to scale the rows (columns), then to scale the
columns (rows). After the first scaling, the largest number in the matrix is 1,
and the second scaling cannot produce numbers that are larger than 1. Therefore,
scaling is achieved after one iteration. Bunch \cite{bunch} developed an
algorithm that equilibrates any symmetric matrix in the infinity norm. More
recently, Ruiz \cite{ruiz} developed another iteration that compares favorably
with Bunch's algorithm. He extended the method to 1- and 2-norms and showed
convergence results for these algorithms as strong as, and based on, those by
Parlett and Landis \cite{parlett-landis} for their algorithms.

Each of these algorithms is iterative and yields a sequence of matrices
converging to a doubly stochastic matrix. A user can terminate the iteration
early to yield an approximately equilibrated matrix; hence these algorithms may
be viewed as approximate equilibration algorithms.

To date, it appears that all scaling algorithms for matrices having signed
elements require access to the elements of the matrix. If $A$ is nonnegative,
the situation is much different; for example, the Sinkhorn-Knopp algorithm
requires only the matrix-vector products (mvp) $A x$ and $A^T x$. For general
matrices, algorithms need at least mvp of the form $|A| x$ (1-norm), $(A \circ
A) x$ (2-norm), or similar expressions, and their transposes. We introduce
approximate scaling algorithms for equilibration in the 2-norm that require only
the mvp $A x$ and $A^T x$, where $x$ is a random vector. Algorithms that compute
the mvp with a random vector have been developed to solve other problems. For
example, Bekas, Kokiopoulou, and Saad \cite{saad} developed a method to estimate
the diagonal elements of a matrix; and Chen and Demmel \cite{chen-demmel}, to
balance a matrix prior to computing its eigenvalues. Our algorithms also have a
connection to the methods of \emph{stochastic approximation}
\cite{book-hard-stoch-approx}.

We want to emphasize that because the algorithms we propose access a matrix
having signed elements only through a sequence of mvp, we cannot expect them to
be faster than, or even as fast as, algorithms that access the elements directly
when applied to matrices for which direct access to the elements is possible and
efficient. Our algorithms are useful only if a matrix has signed elements that
are impossible or inefficient to access directly; it appears there are not
algorithms already available to solve this problem. It is also desirable that
only a small number, relative to the size of the matrix, of mvp are required.

\subsection{Existence and uniqueness} \label{sec:theory}
A matrix has \emph{support} if a positive main diagonal exists under a column
permutation; a matrix having this property is equivalently said to be
\emph{structurally nonsingular} \cite{davis-book}. A square matrix has
\emph{total support} if every nonzero element occurs in the positive main
diagonal under a column permutation.  A matrix has total support if and only if
there exists a doubly stochastic matrix having the same zero pattern
\cite{perfect-mirsky}. A matrix $A$ is \emph{partly decomposable} if there exist
permutation matrices $P$ and $Q$ such that
\begin{equation} \label{eq:pdecomp}
  P A Q = \pmat{E & 0 \\ C & D},
\end{equation}
where $E$ and $D$ are square matrices. A square matrix is \emph{fully
  indecomposable} if it is not partly decomposable. A fully indecomposable
matrix has total support \cite{brualdi-80}. A matrix $A$ is \emph{reducible} if
there exists a permutation matrix $P$ such that $P A P^T$ has the matrix
structure in \eqref{eq:pdecomp}; otherwise, $A$ is \emph{irreducible}. For
convenience, a matrix is said to be \emph{scalable} if it can be equilibrated.

\begin{theorem}[Sinkhorn and Knopp \cite{sinkhorn-knopp}] \label{thm:sk}
  Let $B$ be a nonnegative square matrix.
  \begin{remunerate}
  \item There exist positive diagonal matrices $R$ and $C$ such that $F \equiv R
    B C$ is doubly stochastic---briefly, $B$ is scalable---if and only if $B$
    has total support.
  \item If $B$ is scalable, then $F$ is unique.
  \item $R$ and $C$ are unique up to a scalar multiple if and only if $B$ is
    fully indecomposable.
  \item The Sinkhorn-Knopp iteration yields a sequence of matrices that
    converges to a unique doubly stochastic matrix, for all initial $r,c > 0$,
    if and only if $B$ has support. If $B$ has support that is not total, then
    $R$ and $C$ have elements that diverge.
  \end{remunerate}
\end{theorem}
Parts 1--3 were independently discovered in \cite{pbs-66}.

\begin{theorem}[Csima and Datta \cite{csima-datta}] \label{thm:csima-datta}
  A symmetric matrix is symmetrically scalable if and only if it has total
  support.
\end{theorem}

The necessary and sufficient condition of total support in Theorem
\ref{thm:csima-datta} is identical to that in part 1 of Theorem
\ref{thm:sk}. The necessary part follows directly from part 1, but proving the
sufficiency part requires several steps not needed in the nonsymmetric case.

Section 3 of \cite{knight} discusses the symmetric iteration
\begin{equation} \label{eq:sk-sym}
  x^{k+1} = (B x^k)^{-1}
\end{equation}
for symmetric $B$ and sketches a proof of convergence. Not directly addressed is
that the iterates $x^k$ can oscillate and reducible $B$.

If $B$ is irreducible, this oscillation is straightforward and benign. The
resulting scaled matrix is a scalar multiple of a doubly stochastic one. For
example, suppose $\bar B = 1$ and $x^0 = 2$. Then for $k$ even, $x^k = 2$, and
for $k$ odd, $x^k = 1/2$. In general, if symmetric $B$ is irreducible, $X^k B
X^{k+1}$ converges to a doubly stochastic matrix, while $X^{2k} B X^{2k}$ and
$X^{2k+1} B X^{2k+1}$ converge to scalar multiples of a doubly stochastic
matrix, and these scalars are reciprocals of each other.

Somewhat more complicated is reducible $B$. For example, consider the matrix
$\bar B = \text{diag}(1 \ 2)^T$. If $x^0 = e$, the even iterates remain $e$
while the odd iterates are $v \equiv (1 \ 1/2)^T$. $I \bar B V$ is doubly
stochastic, but $v$ is not proportional to $e$. Moreover, $V \bar B V$ is not
simply a scalar multiple of a doubly stochastic matrix. This nonconvergence is
also benign. A reducible symmetric matrix $B$ can be symmetrically permuted to
be block diagonal with each block irreducible. Hence the equilibration problem
is decoupled into as many smaller problems. We can construct a symmetric
equilibrating vector $x$ from the nonsymmetric equilibrating vectors $r$ and $c$
by setting $x = \sqrt{r c}$. For suppose $r$ and $c$ equilibrate $B$ by $R B
C$. Let ${\cal I}$ be the indices corresponding to an irreducible block. Then
$r({\cal I}) \propto c({\cal I})$ and the block $X({\cal I},{\cal I}) B({\cal
  I},{\cal I}) X({\cal I},{\cal I})$ is doubly stochastic. For $\bar B$, the
symmetric equilibration vector is $\sqrt{e v} = (1 \ 1/\sqrt{2})^T$.

These observations suggest that we should write the symmetric Sinkhorn-Knopp
iteration as
\begin{equation} \label{eq:sk-sym-revised}
  y^{k+1} = (B y^k)^{-1}, \quad x^{k+1} = \sqrt{y^{k+1} y^k}.
\end{equation}
Since $x^k$ does not actually play a role in the iteration, in practice, the
square root operation needs to be applied only after the final iteration to
yield the scaling matrix.

\subsection{Stochastic equilibration}
Our algorithms are based on the Sinkhorn-Knopp iteration. The Sinkhorn-Knopp
iteration performs the mvp $B x$ and $B^T x$ for a nonnegative matrix $B$. If
$A$ is a matrix having signed elements, then $B_{ij} = |A_{ij}|^p$ for $p \ge 1$
for equilibration in the $p$-norm, and so $B$ is not available if one does not
have access to the elements of $A$. The key idea, similar to that in
\cite{chen-demmel}, in our algorithms is to compute $B x$ approximately by using
an mvp with $A$ rather than $B$, where $B \equiv A \circ A$, and similarly for
$B^T x$.

Let $a \in \mathbb{R}^n$. If the elements of the random vector $u \in
\mathbb{R}^n$ have zero mean, positive and finite variance, and are iid, then
$\E (a^T u)^2 = \eta \E a^T a$ for finite $\eta > 0$, where E denotes
expectation. For as $\E u_i u_j = 0$ if $i \ne j$, $\E ( \sum_j a_j u_j )^2 = \E
\sum_j a_j^2 u_j^2 = \eta \sum_j a_j^2$, where $\eta = \E u_j^2 > 0$ is
finite. See \cite{chen-demmel} for more on this and related expectations. We use
this fact to approximate $B x$ by computing the mvp $A X^{1/2} u$:
\begin{equation} \label{eq:AXu_stoch}
  \E (A X^{1/2} u)^2 = \eta ((A X^{1/2}) \circ (A X^{1/2})) e = \eta (A \circ A)
  X e = \eta B x.
\end{equation}

To increase the accuracy of the approximation to $Bx$, one could compute the
mean of multiple mvp $A X^{1/2} u$. Then one could construct an approximate
scaling algorithm by replacing the exact computation $B x$ with this estimate,
and similarly for $B^T x$, in the Sinkhorn-Knopp algorithm. However, the method
of \emph{stochastic approximation} \cite{book-hard-stoch-approx} suggests a
better approach. In stochastic approximation, the exact iteration $x^{k+1} = x^k
+ \omega^k f(x^k)$ is replaced by the stochastic iteration $x^{k+1} = x^k +
\omega^k \tilde f(x^k)$, where $\E \tilde f(x^k) = f(x^k)$ and $x$ is sought
such that $f(x) = 0$. Rather than explicitly average multiple realizations of
$\hat f(x^k)$ at each iteration, the stochastic approximation iteration controls
the relative weight of $\hat f$ through $\omega^k$ and alters the iterate $x^k$
at each evaluation of $\hat f$.

Let $\rho \equiv r^{-1}$, $\gamma \equiv c^{-1}$, and $0 < \omega^k <
1$. Consider the iteration
\begin{align}
  \rho^{k+1} &= (1 - \omega^k) \frac{\rho^k}{\|\rho^k\|_1} +
    \omega^k \frac{B c^k}{\|B c^k\|_1} \label{eq:de1a} \\
  \gamma^{k+1} &= (1 - \omega^k) \frac{\gamma^k}{\|\gamma^k\|_1} +
    \omega^k \frac{B^T r^{k+1}}{\|B^T r^{k+1}\|_1}. \notag
\end{align}
This iteration takes a convex combination of the reciprocal of an iterate and
the Sinkhorn-Knopp update when each is normalized by its 1-norm.
Let $u^k$ and $v^k$ be random vectors as before. For the vector $x$, $(x)^2$ is
the element-wise square. Substituting \eqref{eq:AXu_stoch} into this iteration,
we obtain the stochastic iteration
\begin{align}
  y^k &= (A (C^k)^{1/2} u^k)^2 \notag \\
  \rho^{k+1} &= (1 - \omega^k) \frac{\rho^k}{\|\rho^k\|_1} +
             \omega^k \frac{y^k}{\|y^k\|_1} \label{eq:se1a} \\
  z^k &= (A^T (R^{k+1})^{1/2} v^k)^2 \notag \\
  \gamma^{k+1} &= (1 - \omega^k) \frac{\gamma^k}{\|\gamma^k\|_1} +
               \omega^k \frac{z^k}{\|z^k\|_1} \notag.
\end{align}
We implement this iteration in the {\sc Matlab} function {\tt snbin}.
\begin{lstlisting}[language=matlab,basicstyle=\footnotesize]
     function [r c] = snbin(A,nmv,m,n)
     % Stochastic matrix-free binormalization for nonsymmetric real A.
     %   A is a matrix or function handle. If it is a function handle,
     %     then v = A(x) returns A*x and v = A(x,`trans') returns A'*x.
     %   nmv is the number of forward and transpose matrix-vector
     %     product pairs to perform.
     %   m,n is the size of the matrix. It is necessary to specify
     %     these only if A is a function handle.
     %   diag(r) A diag(c) is approximately binormalized (to a scalar).
       op = isa(A,`function_handle');
       if(~op) [m n] = size(A); end
       r = ones(m,1); c = ones(n,1);
       for(k = 1:nmv)
         % omega^k
         alpha = (k - 1)/nmv;
         omega = (1 - alpha)*1/2 + alpha*1/nmv;
         % rows
         s = randn(n,1)./sqrt(c);
         if(op) y = A(s); else y = A*s; end
         r = (1-omega)*r/sum(r) + omega*y.^2/sum(y.^2);
         % columns
         s = randn(m,1)./sqrt(r);
         if(op) y = A(s,`trans'); else y = (s'*A)'; end
         c = (1-omega)*c/sum(c) + omega*y.^2/sum(y.^2);
       end
       r = 1./sqrt(r); c = 1./sqrt(c);
\end{lstlisting}
Our choice of the sequence $\omega^k$ is based on numerical experiments; the
sequence encourages large changes in $d/\|d\|_1$ when $k$ is small and smaller
changes when $k$ is large.

Iteration \eqref{eq:de1a} forms a linear combination of $\rho^k$ and $B
c^k$. One might consider instead forming a linear combination of $r^k$ and $(B
c^k)^{-1}$. In the iteration we use, a reciprocal is taken after forming a
linear combination of the iterate and a random quantity; in contrast, in this
alternative, it is taken before, and of the random quantity. Consequently, the
stochastic iteration corresponding to this alternative iteration is less stable
than \eqref{eq:se1a}.

A straightforward algorithm for the symmetric problem applies {\tt snbin} to the
symmetric matrix $B$ and then returns $\sqrt{r c}$. But numerical experiments
suggest we can do better. For irreducible matrices, the denominators $\|d^k\|_1$ and
$\|B x^k\|_1$ in the iteration
\begin{equation*}
  d^{k+1} = (1 - \omega^k) \frac{d^k}{\|d^k\|_1} + \omega^k \frac{B x^k}{\|B
    x^k\|_1}
\end{equation*}
remove the benign oscillation we observed in Section \ref{sec:theory};
therefore, adjacent iterates, rather than every other one as in \eqref{eq:de1a},
can be combined in a convex sum. This second approach speeds convergence. But it
is not sufficient when applied to reducible matrices. Numerical experiments
support using the second approach for early iterations, when making progress
quickly is important, and then switching to the first approach to refine the
scaling matrix. We implement this strategy in {\tt ssbin}.
\begin{lstlisting}[language=matlab,basicstyle=\footnotesize]
     function x = ssbin(A,nmv,n)
     % Stochastic matrix-free binormalization for symmetric real A.
     %   A is a symmetric real matrix or function handle. If it is a
     %     function handle, then v = A(x) returns A*x.
     %   nmv is the number of matrix-vector products to perform.
     %   [n] is the size of the matrix. It is necessary to specify n
     %     only if A is a function handle.
     %   diag(x) A diag(x) is approximately binormalized (to a scalar).
       op = isa(A,'function_handle');
       if(~op) n = size(A,1); end
       d = ones(n,1); dp = d;
       for(k = 1:nmv)
         % Approximate matrix-vector product
         u = randn(n,1);
         s = u./sqrt(dp);
         if(op) y = A(s); else y = A*s; end
         % omega^k
         alpha = (k - 1)/nmv;
         omega = (1 - alpha)*1/2 + alpha*1/nmv;
         % Iteration
         d = (1-omega)*d/sum(d) + omega*y.^2/sum(y.^2);
         if (k < min(32,floor(nmv/2))) % Ignore reducibility.
           dp = d;
         else     % This block makes ssbin behave like snbin.
           tmp = dp; dp = d; d = tmp;        % Swap dp and d.
         end
       end
       x = 1./(d.*dp).^(1/4);        % In case B is reducible
\end{lstlisting}
The final line implements the square root in \eqref{eq:sk-sym-revised}.

In most iterative algorithms, a measure of the merit of an iterate that requires
little work to evaluate relative to the work in an iteration influences the
behavior of the algorithm. In our algorithms, any procedure to assess the merit
of an iterate would require additional mvp, likely wasting work. Hence the
parameter values in the loop of each algorithm are fixed independent of problem.

\section{Numerical experiments} \label{sec:numexp}
\begin{figure}[tb]
\centering
\includegraphics[width=5in]{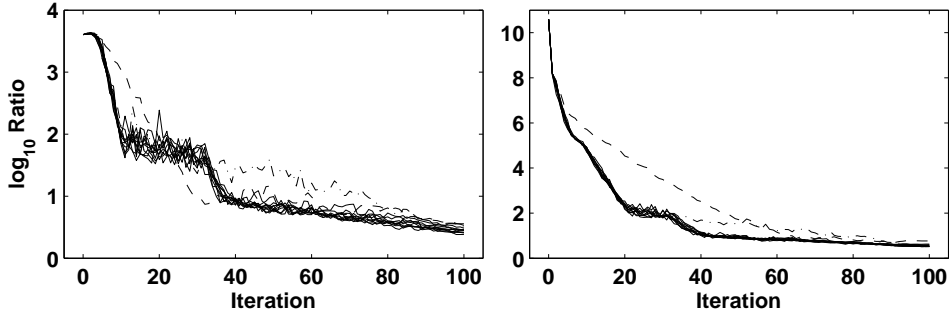}
\caption{Convergence histories for two symmetric problems having sizes 3564
  (left) and 226340 (right). Ten solid lines show individual realizations of the
  algorithm. The dashed line corresponds to {\tt snbin}. The dotted line
  corresponds to not switching to {\tt snbin}-like behavior.}
\label{fig:conv}
\end{figure}

\begin{figure}[tb]
\centering
\includegraphics[width=5in]{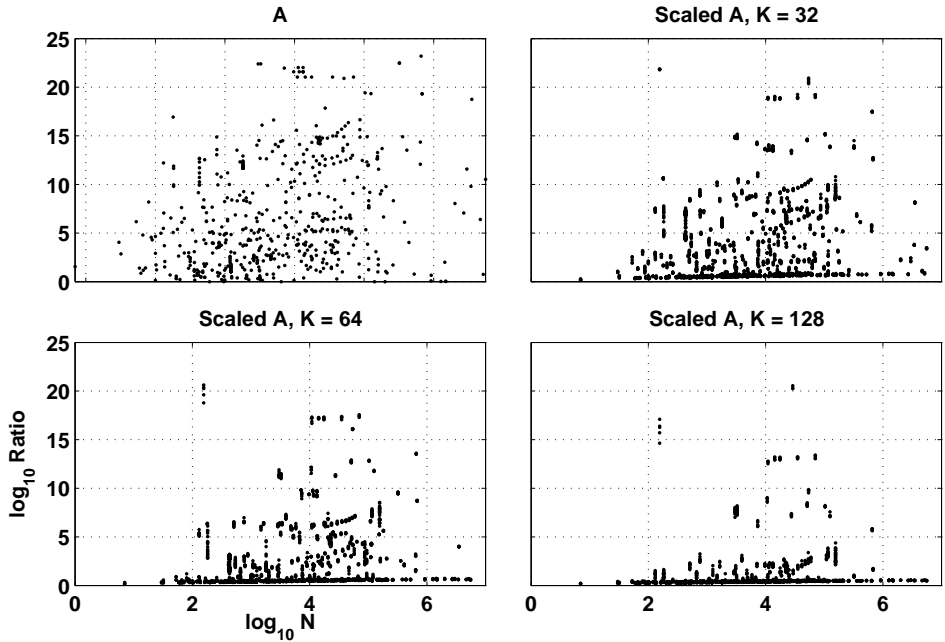}
\caption{Ratio for the original and scaled nonsymmetric matrix vs.~matrix size
  $N$, after the indicated number of iterations, for 741 matrices.}
\label{fig:exp-snbin-a}
\end{figure}

\begin{figure}[tb]
\centering
\includegraphics[width=5in]{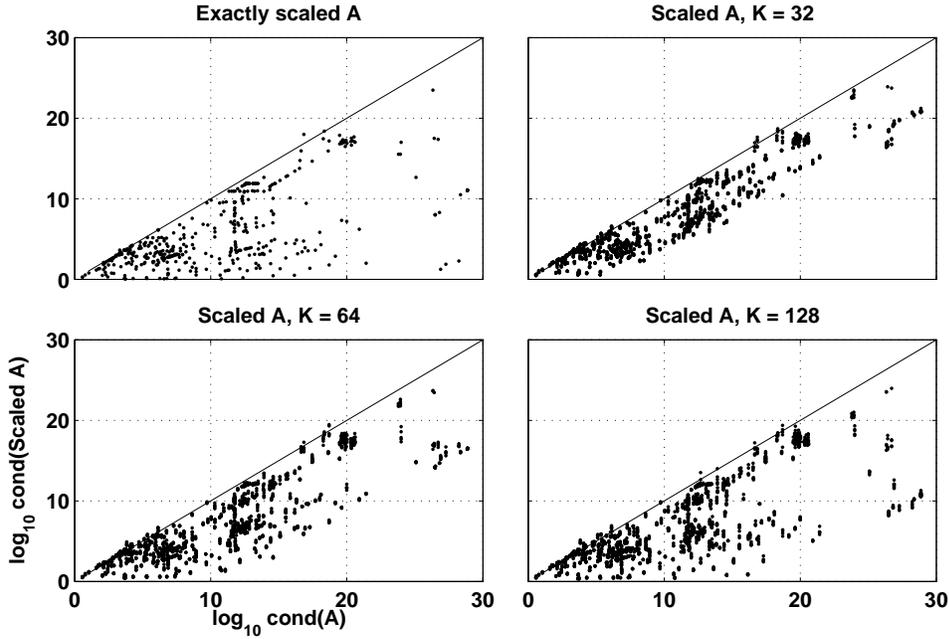}
\caption{Condition number of the scaled nonsymmetric matrix vs.~condition number
  of the original matrix for 519 matrices (matrices having $N \le 2 \times
  10^4$).}
\label{fig:exp-snbin-b}
\end{figure}

\begin{figure}[tb]
\centering
\includegraphics[width=5in]{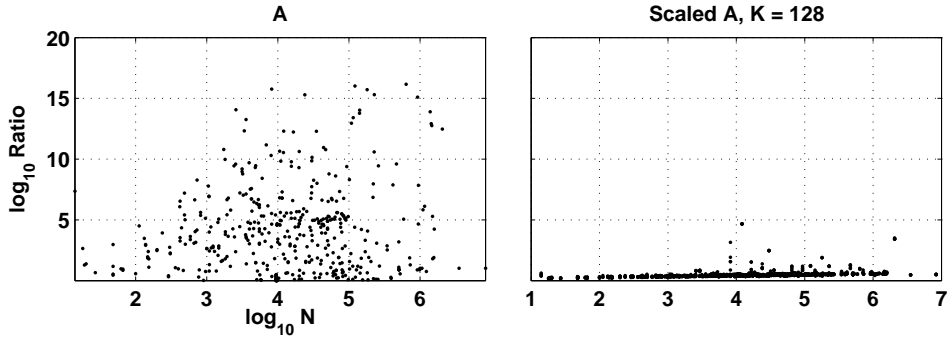}
\caption{Ratios for 466 symmetric matrices. Results for only $K = 128$ are
  shown; trends in $K$ follow those for the nonsymmetric problems.}
\label{fig:exp-ssbin-a}
\end{figure}

\begin{figure}[tb]
\centering
\includegraphics[width=5in]{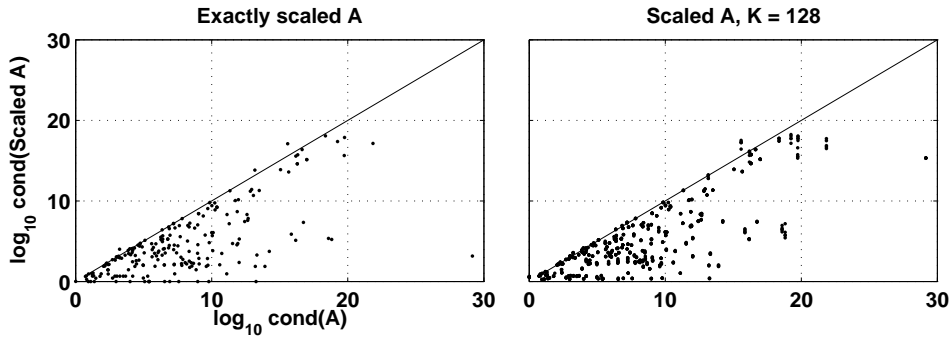}
\caption{Condition numbers for 221 symmetric matrices (matrices having $N \le 2
  \times 10^4$).}
\label{fig:exp-ssbin-b}
\end{figure}

\begin{figure}[tb]
\centering
\includegraphics[width=5in]{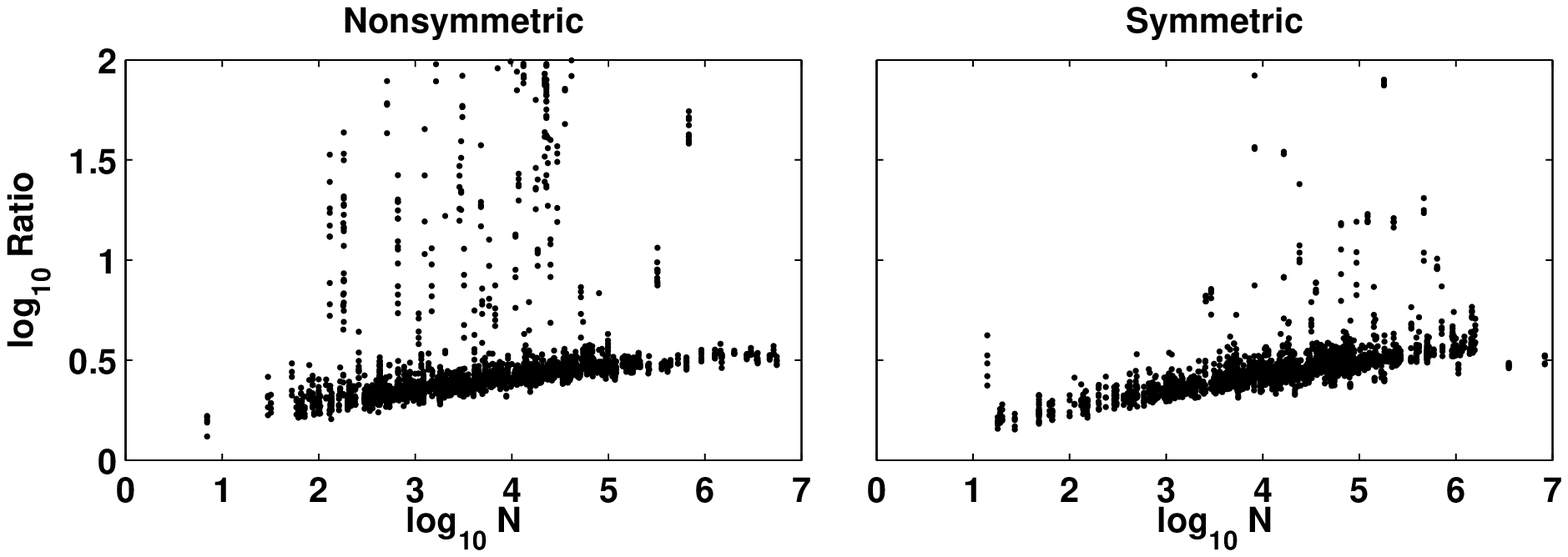}
\caption{A closer look at the ratio as a function of $N$ for $K = 128$
  iterations for nonsymmetric (left) and symmetric (right) matrices.}
\label{fig:exp-N}
\end{figure}

In our numerical experiments, two quantities of the scaled matrices are
measured: condition number if the matrix is not too large; and the ratio of the
largest to smallest row 2-norms (in the nonsymmetric case, row or column,
depending on which gives a larger number), hereafter designated as the
\emph{ratio}.

We test {\tt snbin} and {\tt ssbin} in {\sc Matlab} on matrices in the
University of Florida Sparse Matrix Collection \cite{davis_uf}; these are
obtained by the following queries:
\begin{lstlisting}[language=matlab,basicstyle=\footnotesize]
     index = UFget(`refresh');
     % Symmetric
     sids = find(~index.isBinary & index.numerical_symmetry == 1 &...
                 index.sprank == index.nrows & index.isReal);
     % Square nonsymmetric
     nids = find(~index.isBinary & index.numerical_symmetry < 1 &...
                 index.nrows == index.ncols &...
                 index.sprank == index.nrows & index.isReal);
\end{lstlisting}

First we investigate the behavior of {\tt ssbin} on two problems. Figure
\ref{fig:conv} shows the convergence history, starting with the unaltered
matrix, for 12 runs of {\tt ssbin} on two symmetric problems. The smaller has
size 3564; the larger has size 226340. $\log_{10}$ ratio is used to measure
convergence. The ten closely clustered solid curves correspond to the nominal
algorithm. The dashed curve indicates the slower convergence of simply applying
{\tt snbin}. The dotted curve shows the problem with not eventually switching to
{\tt snbin}-like behavior to address reducibility. The plateau in the solid
curves ends at iteration 32, when the switch is made. The ten solid curves are
closely clustered, indicating the variance in the algorithm's output is small
for any number of requested mvp.

In the performance experiments, for each matrix, the algorithm is run five times
for $K = 32$, $64$, and $128$ iterations. Results are shown in Figures
\ref{fig:exp-snbin-a} and \ref{fig:exp-snbin-b} for nonsymmetric matrices and
\ref{fig:exp-ssbin-a} and \ref{fig:exp-ssbin-b} for symmetric matrices. Figure
\ref{fig:exp-snbin-a} shows that the ratio tends to decrease with $K$, as one
expects. The ratio for the scaled problem, given fixed $K$, grows slowly with
problem size $N$.  Figure \ref{fig:exp-N} investigates this aspect more
closely. It shows details for the case of $K = 128$ iterations for nonsymmetric
(left) and symmetric (right) matrices. Over a range of matrix sizes of more than
six orders of magnitude, the final ratio often ranges from between $1.5$ and
$6$.  Figures \ref{fig:exp-snbin-b} and \ref{fig:exp-ssbin-b} show that the
condition number of the scaled matrix is almost always, and often substantially,
smaller than that of the original matrix: any point that falls below the
diagonal line corresponds to a reduction in condition number. The top-left plots
of Figures \ref{fig:exp-snbin-b} and \ref{fig:exp-ssbin-b} show the condition
numbers of the exactly scaled matrices; the ratios are 1, of course. In the
plots corresponding to the stochastic algorithms, what appears to be a point is
in fact a cluster of the five points resulting from the five separate runs. The
tightness of these clusters again implies that the variance of the outputs of
these algorithms is quite small.

These experiments suggest that {\tt ssbin} and {\tt snbin} are effective
matrix-free approximate equilibration algorithms: a small number---relative to
the size of the matrix---of matrix-vector products is sufficient to
approximately equilibrate the matrix. One application is to scale a matrix whose
elements require too much work to access directly prior to using a
Krylov-subspace iteration to solve a linear system. We recommend performing
approximately 100 iterations, which corresponds to 100 matrix-vector products in
the symmetric case and 200 in the nonsymmetric.


\bibliographystyle{siam}
\bibliography{../mybib}

\end{document}